\documentclass[11pt]{amsart}
\usepackage{amssymb}
\usepackage{palatino}
\usepackage{tikz-cd}
\input amssym.def

\usepackage{xcolor}

\usepackage{amsmath, amsfonts}
\usepackage{amssymb}
\usepackage{amscd}
\usepackage[mathscr]{eucal}
\usepackage{palatino}
\usepackage{theoremref}
\newfont{\cyrr}{wncyr10}

\newcommand{\thmref}[1]{Theorem~\ref{#1}}
\newcommand{\corref}[1]{Corollary~\ref{#1}}
\newcommand{\propref}[1]{Proposition~\ref{#1}}
\newcommand{\lemref}[1]{Lemma~\ref{#1}}

\newcommand{\rmkref}[1]{Remark~\ref{#1}}

\newcommand{\Z}{{\mathbb Z}}

\def\({\left(}
\def\){\right)}
\def\[{\left[}
\def\]{\right]}

\def\N{\mathbb{N}}
\def\R{\mathbb{R}}
\def\C{\mathbb{C}}
\def\Q{\mathbb{Q}}
\def\K{\mathbb{K}}
\def\H{\mathbb{H}}
\def\cO{\mathcal{O}}

\def\cD{\mathcal{D}}

\def\cI{\mathcal{I}}

\def\ap{\alpha_{p_0}}
\def\bp{\beta_{p_0}}

\def\fp{\mathfrak{p}}
\def\fd{\mathfrak{d}}
\def\fD{\mathfrak{D}}
\def\fn{\mathfrak{n}}
\def\fm{\mathfrak{m}}

\def\fa{\mathfrak{a}}
\def\bk{\textbf{k}}
\def\bf{\textbf{f}}
\def\sa{\alpha}
\def\sb{\beta}
\def\w{\omega}
\def\afp{\alpha_{\fp}}
\def\ap{\alpha_p}
\def\bp{\beta_p}
\def\bfp{\beta_{\fp}}

\def\e{\epsilon}
\def\sg{\sigma}

\renewcommand{\mod}{{\, \rm mod \, }}

\newtheorem{thm}{Theorem}

\newtheorem{lem}[thm]{Lemma}
\newtheorem{cor}[thm]{Corollary}
\newtheorem{prop}[thm]{Proposition}
\newtheorem{rmk}{Remark}[section] 
\newtheorem{defn}{Definition}

\newtheorem{conj}{Conjecture}

\title[Prime factors and radicals of Fourier coefficients]{On 
the number of prime divisors and radicals of non-zero 
Fourier coefficients of Hilbert cusp forms}

\author{SUNIL L NAIK}
\address{The Institute of Mathematical Sciences, 
A CI of Homi Bhabha National Institute, 
CIT Campus, Taramani, 
Chennai 600 113, 
India.}
\email{sunilnaik@imsc.res.in}

\begin{document}

\subjclass[2010]{11F11, 11F30, 11N56}

\keywords{Prime divisors, Radicals, Fourier coefficients of eigenforms }

\maketitle

\begin{abstract}
In this article, we derive lower bounds for the number of distinct prime divisors of 
families of non-zero Fourier coefficients of non-CM primitive cusp forms and more
generally of non-CM primitive Hilbert cusp forms.  
In particular, for the Ramanujan $\Delta$-function, we show that for any $\e > 0$, 
there exist infinitely many natural numbers $n$ such that $\tau(p^n)$ has at least
$$
2^{(1-\e) \frac{\log n}{\log\log n}}
$$
distinct prime factors for almost all primes $p$. This  improves and refines the
existing bounds. 
We also study lower bounds for absolute norms of radicals of 
 non-zero Fourier coefficients of Modular forms alluded to above.
\end{abstract}

\section{Introduction}
Let $\tau(n)$ denotes the Ramanujan tau function defined by
$$
\sum_{n=1}^{\infty} \tau(n) q^n = q \prod_{n=1}^{\infty} (1-q^n)^{24}, 
$$
where $q=e^{2\pi i z}$ for $z \in \H = \{ x + iy  ~|~ x, y \in \R,~ y >0 \}$.
In \cite{SR}, Ramanujan conjectured the following properties of $\tau(n)$:
\begin{itemize}
	\item $\tau(n)$ is multiplicative i.e. $\tau(mn)= \tau(m) \tau(n), \phantom{m} \text{if } (m,n)=1$.
	\item $\tau(p^{m+1}) = \tau(p) \tau(p^m) - p^{11} \tau(p^{m-1})$, for any prime $p$ and $m \in \N$.
	\item $|\tau(p)| \leq 2 p^{11/2}$~ for all primes $p$.	
\end{itemize}
First two properties were proved by Mordell while the third property was proved by Deligne.
Ramanujan also conjectured that $\tau(n) \equiv 0 \,(\mod 691)$ for all $n$ 
except for a set of density $0$. This conjecture
was proved by Watson \cite{GW}.  Later Serre \cite{JP} proved a stronger assertion, namely 
 for any natural number $d$, $\tau(n) \equiv 0 (\mod d)$ except for a set of density $0$.
 This was further strengthened by Ram Murty and Kumar Murty in \cite{MMp} 
where they proved that $\tau(n) \equiv 0 \,(\mod d^{t})$ for almost all $n$, 
where $t = [\delta \log\log n]$ 
and $\delta >0$ is a constant depending only on $d$. Results of Serre and Murty-Murty 
also holds for normalized Hecke eigenforms of integral weight and higher level 
having integer Fourier coefficients.

Before moving further let us fix few  conventions for densities.  Let $P(n)$ be a property
for natural numbers $n$. We say $P$ holds for almost all $n$ if
$$
\lim_{x \to \infty} \frac{\#\{n \le x : P(n)\,\, \rm {holds}\}}{x}  = 1.
$$

On the other hand , we say $P$ holds for almost all primes if
$$
\lim_{x \to \infty} \frac{\#\{p \le x : \,\, p \,\, {\rm prime}, \,\, P(p)\,\, \rm {holds}\}}{\frac{x}{{\log x}}} = 1.
$$

\smallskip

For any integer $n$, let $\w(n)$ denotes the number of distinct prime factors 
of $n$ with the convention that $\w(0)=\w(\pm 1) = 0$.
It is an interesting albeit difficult theme in Number theory to study the  behaviour of  $\w(f(n))$ for 
arithmetic functions $f$ taking integral values. When $f(n)=n$, Hardy 
and Ramanujan \cite{HR} proved that $\w(n)$ has normal order $\log\log n$. 
The case when $f(n)$'s are Fourier coefficients of Hecke eigenforms
is considered in \cite{MMp}. Assuming Generalized 
Riemann hypothesis,  Murty-Murty  \cite{MMp} proved that $\w(\tau(n))$ 
has normal order $\frac{1}{2} (\log\log n)^{2}$. More precisely for any $\e > 0$, 
\begin{equation}\label{normal}
\(\frac{1}{2}-\e\) (\log\log n)^2 < \w(\tau(n)) < \(\frac{1}{2}+\e\) (\log\log n)^2
\end{equation}
holds for almost all $n$.

\smallskip
On the other hand,  Luca and Shparlinski \cite{LS} showed that for any $\e > 0$,
\begin{equation}\label{LSh}
\w(\tau(n)) \geq (1-\e) \log\log n
\end{equation}
for infinitely many natural numbers $n$. 
In this article, we first prove the following.
\begin{thm}\label{tau}
For any $\e >0 $, there exists infinitely many natural
numbers $n$ such that
$$
\w(\tau(p^n)) \geq 2^{(1-\e) \frac{\log n}{\log\log n}}
$$
for almost all primes $p$.
\end{thm}

\begin{rmk}\label{rmk1.1}
From \thmref{tau}, it follows that
$$
\w(\tau(m)) \geq 2^{(1-\e)\frac{\log\log m}{\log\log\log m}}
$$
for infinitely many natural numbers $m$. This improves the bound \eqref{LSh} 
deduced by Luca and Shparlinski.
\end{rmk}

\begin{rmk} 
It is also worthwhile to compare our bounds with that in \eqref{normal}. 
The results of Murty-Murty shows that the bound
$$
\w(\tau(n)) < \(\frac{1}{2}+\e\) (\log\log n)^2
$$
holds for almost all $n$, subject however to Generalized Riemann Hypothesis.
On the other hand, we unconditionally produce a sparse infinite subsequence $\{n_i\}$  
such that
$$
\w(\tau(n_i)) \geq 2^{(1-\e)\frac{\log\log n_i}{\log\log\log n_i}} \,.
$$
This in particular shows that the upper bound in \eqref{normal}
cannot hold for all but finitely many $n$.
\end{rmk}
\begin{rmk} 
	If Lehmer's question is answered in affirmative, namely $\tau(n) \ne 0$
		for all natural numbers $n$, then \thmref{tau} will hold for all primes.
\end{rmk}

We also prove following general theorem .

\begin{thm}\label{EHE}
Let $f$ be a non-CM primitive cusp form of weight $k$, level $N$ with trivial character
having Fourier coefficients $\{ a_f(n) ~|~ n \in \N \}$. Also 
$\K_f = \Q(\{a_f(n) ~|~n\in \N \})$ be the associated number field
and $\cO_{\K_{f}}$ be its ring of integers.
Let $\w_f(a_f(n))$ be the number of distinct prime ideals dividing 
$a_f(n) \cO_{\K_f}$. Then there exists a subset $P_f$ of primes of density one
such that for any $p \in P_f$ and any $\e > 0$,  there exists infinitely many natural numbers $n$ such that
$$
\w_f(a_f(p^n)) \geq 2^{(1-\e) \frac{\log n}{\log\log n}}.
$$
Here by a `primitive cusp form', we mean normalized Hecke eigen cusp form 
lying in the newform space. As before, we follow the convention that $\omega_f(u)=0$ if $u=0$ or a unit in $\mathcal{O}_{\mathbb{K}_f}$. 
\end{thm}

We then apply our results to derive  lower bounds for absolute norms of 
radicals of certain Fourier coefficients of non-CM primitive cusp forms  
$f$. Let $Q(n)$ denotes the radical of $n$, i.e. 
$Q(n) = \prod_{p | n} p$. More generally, for an integral ideal 
 $\cI$ of the ring of integers $\cO_{K}$ of a number field 
 $K$, radical of $\cI$  is given by
$$
rad(\cI) = \prod_{\fp ~|~ \cI} \fp
$$
where $\fp$ runs over distinct prime ideals containing $\cI$.
For any $\sa \in \cO_{K}$, let  $Q_K(\sa)$ denotes the absolute norm of radical of 
$\sa\cO_{K}$ given by
$$
Q_K(\sa) = \prod_{\fp | \sa \cO_K} N_K(\fp)
$$
where $\fp$ runs over distinct prime ideal divisors of $\sa \cO_K$ and $N_K$ 
denotes the absolute norm on $K$. Here we have the convention that $Q_K(0) = 1$.

As a consequence of our above theorems, we have the following corollary.
\begin{cor}\label{cor1}
For any  $\e > 0$, there exist infinitely many natural numbers $n$ such that
$$
\log Q(\tau(p^n)) \geq 2^{(1-\e) \frac{\log n}{\log\log n}},
$$
for almost all primes $p$. More generally, if $f$ is as in \thmref{EHE}, then for any $p \in P_f$ and any $\e > 0$, there exist infinitely many natural numbers $n$ such that
$$
\log Q_f(a_f(p^n)) \geq 2^{(1-\e) \frac{\log n}{\log\log n}}.
$$
Here $Q_f$ is a simplified notation for $Q_{K_f}$.
\end{cor}

\begin{rmk}
\corref{cor1} also follows for the work of Stewart \cite[Eq. 16]{Strad}.
\end{rmk}

We then extend our study to the context of Hilbert modular forms. The detailed statements
of our results in the realms of Hilbert modular forms are described in the next section.

\smallskip

Finally, a refined version of ABC-conjecture proposed by Frey naturally enters 
our set up. For instance,  we  improve \corref{cor1} by appealing
to this conjecture. We invite the reader to look in the next sections for
the statement of Frey's conjecture as well the ensuing consequences.

\smallskip

Let us end the introduction with  a brief outline of our article.
In  \S2, we list the statements of our results for Hilbert modular forms as
 well as the applications of Frey's ABC-conjecture.
 In  \S3 , we recall the relevant notions and notations for elliptic and Hilbert modular forms
 as well as  heights of an algebraic number. We shall also indicate the
 prerequisites for Lucas sequences 
 and the  statement of Frey's ABC-conjecture. Finally
 proof of all the results are given in \S4 -- \S7.

\section{Statements of  Results for Hilbert Modular forms}

In this section, we  state our results for Hilbert Modular  forms.
Please see the next section for description and clarifications of
the terms and terminologies in these statements.

\begin{thm}\label{thm1}
Let $F \subset \R$ be a Galois extension of $\Q$ of odd degree with ring of integers $\cO_F$
and $\fn$ be an integral ideal of $\cO_F$. 
Also let $\bf$ be a non-CM primitive Hilbert cusp form of weight 
$2\bk = \underset{n~ times}{\underbrace{(2k,\cdots,2k)}}$ and level 
$\fn$ with trivial character. Further suppose that $\bf$ has Fourier coefficients $\{ c(\fm)\}$ 
at each integral ideal $\fm \subset \cO_F$. As before, $\K_{\bf}= \Q(\{ c(\fm) ~|~ \fm \subset \cO_F\})$,
$\cO_{\K_{\bf}}$ is the ring of integers of $\K_{\bf}$
and $\w_{\bf}(c(\fm))$ denotes the number of distinct prime ideals dividing $c(\fm)\cO_{\K_{\bf}}$. Then there exist a subset $P_{\bf}$ of prime ideals in $\cO_{F}$ of density one such that for any $\fp \in P_{\bf}$ and any $\e > 0$, there exists infinitely many natural numbers $n$ such that
$$
\w_{\bf}(c({\fp}^n)) \geq 2^{(1-\e) \frac{\log n}{\log\log n}}.
$$
Here $\omega_{\textbf{f}}(u)=0$ if $u=0$ or a unit in $\mathcal{O}_{\mathbb{K}_{\textbf{f}}}$.
\end{thm}
 
As a consequence of \thmref{thm1}, we derive the following corollary. 
\begin{cor}\label{cor2}
Let $\bf$ be as in \thmref{thm1}. For any  $\fp \in P_{\bf}$ and any $\e > 0$, there exist infinitely many natural numbers $n$ such that
$$
\log Q_{\bf}(c({\fp}^n)) \geq 2^{(1-\e) \frac{\log n}{\log\log n}}.
$$
Here $Q_{\bf}$ is a simplified notation of $Q_{\K_{\bf}}$.
\end{cor}
If we assume Frey's refined ABC-conjecture (see \S3.5), 
we can improve \corref{cor1}. In fact, we prove the following theorem.

\begin{thm}\label{proprad-int}
Let $K$ be a number field and $A, B \in \cO_K \setminus \{0\}$ be such that $A/B$ is not a 
root of unity. Also let $(A, B)=1$ and $h(A) \geq h(B)$. 
Assuming Frey's refined ABC-conjecture as mentioned in \S3.5, for any $\e > 0$, we have 
$$
Q_K(A^n -B^n) \geq \frac{c(\e, \Delta_K)}{Q_K(AB)} ~ e^{(1-\e) [K : \Q] h(A) n} 
$$
for all natural numbers $n$, where $c(\e, \Delta_K)$ is a positive constant
depending on $\e$ and discriminant $\Delta_K$ of $K$. Here $h(\cdot)$ denotes the absolute logarithmic height of an algebraic number (see \S3).

\end{thm}
\begin{rmk}\label{rmkrad-int}
Let $A, B$ be non-zero algebraic integers such that $(A, B)=1$, $h(A) \geq h(B)$ 
and $A/B$ is not a root of unity. 
Suppose that $K_0$ is a number field which contains $AB, A+B$ and 
hence contains the sequence $\frac{A^n - B^n}{A-B}$ for
all $n \in \N$. Also suppose that Frey's refined ABC-conjecture holds.
Then it follows from \thmref{proprad-int} that for any $\e > 0$ 
and all natural numbers $n$ , we have
\begin{equation*}
Q_{K_0}\(\frac{A^n-B^n}{A-B} \) \geq \frac{c(\e, \Delta_K)}
{Q_{K_0}( AB(A-B)^2 )} ~ e^{(1-\e) [K_0 : \Q] h(A) n},
\end{equation*}
where $K$ is the field $K_0(A)$.
\end{rmk}
When $A, B \in \cO_{K}$ are not necessarily co-prime, we have the following theorem.
\begin{thm}\label{proprad-alg}
Let $K$ be a number field and $A, B \in \cO_K \setminus \{0\}$ be such that $A/B$ is not a root of unity.  
Also let $K_1$ be an extension of $K$ of smallest degree such that $(A,B)$ is 
a principal ideal in $K_1$. Assuming Frey's refined ABC-conjecture, for any $\e > 0$, we have 
$$
Q_K(A^n -B^n) \geq \frac{c(\e, \Delta_{K_1})}{Q_K(AB)}~ e^{(1-\e) [K : \Q] h\(\frac{A}{B}\) n}
$$
for all sufficiently large natural numbers $n$ 
depending only on $\e$ and degree of $\Q(A/B)$ over $\Q$.
\end{thm}

\begin{rmk}\label{rmkrad-alg}
Let $A, B$ be non-zero algebraic integers such that 
$A/B$ is not a root of unity. 
Suppose that $K_0$ is a number field which contains algebraic integers $AB, A+B$ and hence
contains the sequence $\frac{A^n - B^n}{A-B}$ for
all $n \in \N$. Let $K_1$ be an extension of $K_0(A)$ of smallest degree 
such that $(A,B)$ is principal ideal in $K_1$. Also suppose that Frey's refined ABC-conjecture holds.
Then it follows from \thmref{proprad-alg} that for any $\e > 0$, we have
$$		
Q_{K_0}\(\frac{A^n-B^n}{A-B} \) \geq 
\frac{c(\e, \Delta_{K_1})}{Q_{K_0}( AB(A-B)^2 )} ~ e^{(1-\e) [K_0 : \Q] h\(\frac{A}{B}\) n}
$$
for all sufficiently large natural numbers $n$ depending only on $\e$ and 
the degree of $\Q(A/B)$ over $\Q$.
\end{rmk}

Before stating the next theorem, let us fix some notations. For any
primitive cusp form $f$ of weight $k$ and level $N$ with trivial character, 
for any rational prime $p \nmid N$, let $\ap, \bp$ denote the roots of the 
polynomial $x^2 - a_f(p)x + p^{k-1}$.  Similarly for 
primitive Hilbert cusp form~$\bf$ of weight $2\bk$ and level $\fn$ with trivial
character (see \S3), let $\afp, \bfp$ denote the roots of the polynomial
$x^2 - c(\fp)x +  N_F(\fp)^{2k-1}$, where $\fp$ is a prime ideal in $F$ 
with $\fp \nmid \fn \fD_F$.  Here $\fD_F$ is the different ideal of $F$. Set
$$
\gamma_{f,p} = \frac{\ap}{\bp}  
\phantom{mm}\text{and}\phantom{mm} 
\gamma_{\bf, \fp} = \frac{\afp}{\bfp}, 
$$
where $p$ and $\fp$ be as before.  We shall denote the discriminant of $\K_f(\ap)$ 
(respectively $\K_{\bf}(\afp)$) by $\Delta_{f, p}$ (respectively by $\Delta_{{\bf}, {\fp}}$).

When Fourier coefficients are rational integers, applying
\rmkref{rmkrad-int}, we now have the following theorem.
\begin{thm}\label{corrad-int}
Suppose that Frey's refined ABC-conjecture is true. Let $\e > 0$ be arbitrary. Then we have the following.
\begin{itemize}
\item 
For almost all primes $p$ and for all natural numbers $n$, we have 
$$
Q(\tau(p^n)) ~\ge~~ \frac{c(\e, \Delta_{\tau, p})}{Q({\tau(p)}^2 - 4 p^{11})} 
~~p^{\(\frac{11}{2} - \nu_{\tau, p} -\e\) (n+1)~+~\delta_{\tau, p}}~,
$$
where $\nu_{\tau, p} = \nu_p(\tau(p))$, $p$-adic valuation of $\tau(p)$ and 
$$
\delta_{\tau, p} ~=~
\begin{cases}
1 &\text{if } \nu_{\tau, p} \ne 0,\\
-1 & \text{otherwise}.
\end{cases}
$$

\item 
More generally, if $f$ is as in \thmref{EHE} with integer Fourier coefficients,
then for almost all primes $p$ and for all natural numbers $n$, we have
$$
Q(a_f(p^n)) ~\ge~~  \frac{c(\e, \Delta_{f, p})}{Q(a_f(p)^2 - 4p^{k-1})}
~~ p^{(1-\e)\(\frac{k-1}{2} - \nu_{f, p}\) (n+1)~+~ \delta_{f, p}}~, 
$$
where $\nu_{f, p} = \nu_p(a_f(p))$ and
$$
\delta_{f, p} ~=~
\begin{cases}
1 &\text{if } \nu_{f, p} \ne 0,\\
-1 & \text{otherwise}.
\end{cases}
$$

\item 
Finally, if $\bf$ is as in \thmref{thm1} with integer Fourier coefficients,
then for almost all prime ideals~$\fp$ in $\cO_{F}$ and for all natural numbers $n$,
we have
$$
Q(c({\fp}^n)) \geq \frac{c(\e, \Delta_{\bf, \fp})}{Q(c(\fp)^2 - 4 N_F(\fp)^{2k-1})}
~~ {N_F(\fp)}^{(1-\e)\(\frac{2k-1}{2} - \frac{\nu_{\bf, \fp}}{\ell } \)(n+1)
~+~ \frac{\delta_{\bf, \fp}}{\ell}}~, 
$$
where $N_F(\fp)= p^{\ell}$ and  $\nu_{\bf, \fp} = \nu_p(c(\fp))$. Here 
$$
\delta_{\bf, \fp} ~=~
\begin{cases}
1 &\text{if } \nu_{\bf, \fp} \ne 0,\\
-1 & \text{otherwise}.
\end{cases}
$$
\end{itemize}
\end{thm}
The lower bound in \thmref{corrad-int} can be compared to the
upper bound in the Remark below.
\begin{rmk}
By Deligne's bound, we have
$$
|\tau(p^n)| \leq (n+1)~~ p^{\frac{11}{2} n} ~,
$$
for all natural numbers $n$. 
Thus trivially we have $Q(\tau(p^n)) \leq (n+1)~~ p^{\frac{11}{2} n}$. 
Suppose $\tau(p)\neq 0$ and $\nu_{\tau, p} \neq 0$, then we have
$$
Q(\tau(p^n)) \leq \frac{|\tau(p^n)|}{p^{n\nu_{\tau, p}-1}} \leq (n+1)
~~ p^{\(\frac{11}{2}- \nu_{\tau, p} \)n ~+~1}
$$
for all natural numbers $n$.  

When $f$ is as in \thmref{EHE}
with integer Fourier coefficients, we have 
$Q(a_f(p^n)) \leq (n+1)~~p^{\frac{k-1}{2} n}$. Suppose $a_f(p) \neq 0, \nu_{f, p} \neq 0$, then we have
$$
Q(a_f(p^n)) \leq (n+1)~~p^{\(\frac{k-1}{2}- \nu_{f, p} \)n ~+~1}
$$
for all natural numbers $n$.

When $\bf$ is as in \thmref{thm1} with integer Fourier coefficients,
using Blasius's bound (see \S3), we get $Q(c({\fp}^n)) \leq (n+1)~~ N_F(\fp)^{\frac{2k-1}{2} n}$. Suppose $c(\fp) \neq 0, \nu_{\bf, \fp} \neq 0$, then we have
$$
Q(c({\fp}^n)) \leq (n+1)~~ N_F(\fp)^{\(\frac{2k-1}{2}- \frac{\nu_{\bf, \fp}}{\ell} \)n ~+~\frac{1}{\ell}}
$$
for all natural numbers $n$. Here $N_F(\fp)= p^{\ell}$.
\end{rmk}
When Fourier coefficients are not necessarily rational integers,
applying \rmkref{rmkrad-alg}, we deduce the following result.

\begin{thm}\label{corradalg}
Suppose that Frey's refined ABC-conjecture is true. Let $\e > 0$ be arbitrary. Then we have the following.
\begin{itemize}
\item If $f$ is as in \thmref{EHE}, then for 
almost all primes~$p$, we have
$$
Q_f(a_f(p^n)) \geq c(\e, f, p)~~ e^{(1-\e)[\K_f :\Q] h\( \gamma_{f, p}\)n}
$$
for all sufficiently large natural numbers $n$ depending on $\e$ and degree of $\Q(\gamma_{f,p})$
over $\Q$. 
\item For a Hilbert cusp form $\bf$ as in \thmref{thm1}, for almost all
primes ideals $\fp$, we have 
$$
Q_{\bf}(c({\fp}^n)) \geq c(\e, \bf, \fp) ~~e^{(1-\e) [\K_{\bf} : \Q] h\( \gamma_{\bf, \fp} \)n}
$$
for all sufficiently large natural numbers $n$ depending on $\e$ 
and degree of $\Q(\gamma_{\bf, \fp})$
over $\Q$.
\end{itemize}
\end{thm}

\begin{rmk} Let us compare the lower bounds in \corref{cor1} and 
\thmref{corrad-int}, say for the function $\tau$. Let $\e >0$ be a real number and
 $P_{\tau}$ be the set of prime numbers $p$ such that $\gamma_{\tau, p}$ 
 is not a root of unity. Fix a prime $p \in P_{\tau}$. From the proof of \thmref{corrad-int}, it follows that
\begin{equation}\label{logQc}
\log Q(\tau(p^n)) ~\geq~ c(p) n
\end{equation}
for all sufficiently large natural numbers $n$. Here $c(p)$ is a positive constant depending on $p$.
On the other hand, from the proof of \rmkref{rmk1.1}, \corref{cor1} 
holds for prime $p \in P_{\tau}$ and hence we have
\begin{equation}\label{logQ}
\log Q(\tau(p^n)) \geq n^{(1-\e)\frac{\log 2}{\log\log n}}
\end{equation}
for infinitely many natural numbers $n$. So the conditional lower 
bound in \eqref{logQc} is stronger than that of in~\eqref{logQ}.
\end{rmk}

\medskip

\section{Notations and Preliminaries}

\medskip

\subsection{Prerequisites from elliptic modular forms}
Let $f$ be a primitive cusp form of weight $k$, level $N$ with trivial character 
having Fourier coefficients $\{a_f(n) ~|~ n \in \N\}$. 
By Deligne's bound, we can write $a_f(p) = 2p^{\frac{k-1}{2}} cos(\theta_p)$, $\theta_p \in [0, \pi]$.
Now we will state the Sato-Tate conjecture, which is now a theorem due to the 
pioneering work of Clozel, Harris, Taylor, Shepherd-Barron, 
Barnet-Lamb and Geraghty \cite{BGHT,CHT,HST,RT}.
\begin{thm}\label{ST}
Let $f$ be a non-CM primitive cusp form of weight $k$, level $N$. For any $0 \leq c \leq d \leq \pi$, we have
\begin{equation*}
\lim_{x \to \infty} \frac{\#\{ p \leq x : \theta_p \in [c,d]\}}{\#\{p \leq x\}} = \frac{2}{\pi} \int_{c}^{d} \sin^2(\theta) d\theta.
\end{equation*}
\end{thm}
The Sato-Tate theorem is also known in the context of Hilbert cusp forms from the work of Barnet-Lamb, Gee and Geraghty \cite{BGG}.

We also need the following lemma (see \cite[Lemma 2.2]{KRW}, \cite[Lemma 2.5]{MModd})
\begin{lem}\label{afpn0}
Let $f$ be a primitive cusp form of weight $k$, level $N$. For all $p$ sufficiently large, either $a_f(p) = 0$ or $a_f(p^n) \neq 0,~ \forall n \geq 1$.
\end{lem}	
For any prime $p \nmid N$ and natural number $n$, we have the following recurrence formula
$$
a_f({p}^{n+1}) = a_f(p) a_f(p^{n}) - p^{k-1} a_f(p^{n-1}).
$$
Thus for all $n \ge 0$, we have
\begin{equation}\label{apL}
a_f(p^{n}) = \frac{ \sa_{p}^{n+1} -  \sb_{p}^{n+1}}{\ap - \bp} ~, 
\end{equation}
where $\ap, \bp$ are the roots of the polynomial $x^2 -a_f(p)x+ p^{k-1}$.
From \eqref{apL} and \lemref{afpn0}, we note that for all $p$ sufficiently large,
 either $a_f(p) = 0, \pm 2p^{\frac{k-1}{2}}$ (when $\sa_p = \sb_p$) or 
 $\gamma_{f, p} = \sa_p/\sb_p$ is not a root of unity. Also note that 
 $a_f(p) = 0, \pm 2p^{\frac{k-1}{2}}$ is equivalent to saying that 
 $\theta_p \in \{0, \pi/2, \pi\}$. Hence from The Sato-Tate theorem 
 (see also Serre \cite[Cor. 2, \S7]{Sr}) and \lemref{afpn0}, 
 it follows that (see \cite[Prop. 2.1]{BDGL}).
\begin{lem}\label{gunity}
Let $f$ be a non-CM primitive cusp form of weight $k$, level $N$. The 
set $P_f$ of primes $p \nmid N$ such that $\gamma_{f, p}$ is not a 
root of unity has density equal to $1$.
\end{lem} 

\subsection{Prerequisites from Hilbert cusp forms}
We summarize some basic properties of Fourier coefficients of Hilbert 
modular forms (see \cite{Sh} for further details). 
Let $F \subset \R$ be a Galois extension of $\Q$ of odd degree $n$,
$\cO_F$ be its ring of integers and $\fn$ be an integral ideal of $\cO_F$. 
Let $h_F$ be the narrow class number of $F$ 
and $\{t_v\}_{v=1}^{h_F}$ be a complete set of representatives of the narrow 
class group. For each $t_v$, define a subgroup $\Gamma_v(\fn)$ of $GL_2(F)$ by
$$
\Gamma_v(\fn) = \left\{ \begin{pmatrix} a & {t_v}^{-1}b \\ t_v c & d \end{pmatrix} 
                 \in GL_2(F) ~~|~~ a,d \in \cO_F, b \in \fD_F^{-1}, c \in  \fn\fD_F, 
                 ad- bc \in \cO_F^{\times} \right\},
$$
where $\fD_F$ is the different ideal of $F$. \newline

A Hilbert modular form $f_v$ of weight $\bk = (k_1,...,k_n)$ with respect to 
$\Gamma_v(\fn)$ has a Fourier expansion as follows
$$
f_v(z) = \sum_{\xi \in F} a_v(\xi) e^{2 \pi i \xi z},
$$
where $\xi$ runs over all totally positive elements in $t_v^{-1}\cO_F$, $\xi = 0$ and $z \in \H^n$. 

We write $\bf = (f_1,...,f_{h_F})$, where $f_v$'s $(v=1,...,h_F)$ are Hilbert modular forms of 
weight $\bk$ with respect to $\Gamma_v(\fn)$. From now on, by a primitive 
Hilbert cusp form $\bf$ of weight $\bk$ and level $\fn$ with trivial character, 
we mean $\bf$ is a normalized Hecke Hilbert cusp eigenform lying in the newform space. 
Let $\fm \subseteq \cO_F$ be an integral ideal. Write 
$\fm = \xi t_v^{-1} \cO_F$ with a totally positive element $\xi \in F$. Then the 
Fourier coefficients of $\bf$ for such $\fm$ are
$$
c(\fm, \bf)= c(\fm) = {N(\fm)}^{\frac{k_0}{2}} a_v(\xi) {\xi}^{-\frac{\bk}{2}}
$$
where $\bk = (k_1,...,k_n) \in \N^n$ and $k_0 = \max\{k_1,...,k_n\}$.
If $\fm$ is not an integral ideal, we put $c(\fm)=0$. 
We shall further 
assume that $k_1 = ... = k_n$ and are even. 
Hence from here onwards the 
primitive Hilbert cusp forms considered are of weight $2\bk = (2k,...,2k)$. 

We require the following result on the Fourier coefficients of primitive Hilbert cusp 
forms proved by Shimura \cite{Sh}.
 
\begin{prop}\label{shi}
Let $\bf$ be a primitive Hilbert cusp form of weight $2\bk$ and level $\fn$
with trivial character. Then
\begin{itemize}
	\item For any co-prime ideals $\fm_1, \fm_2 \subset \cO_F$, one has 
	      $$
	      c(\fm_1 \fm_2) = c(\fm_1) c(\fm_2)
	      $$
	\item For $n \in \N$ and any prime $\fp \nmid \fn \fD_F$, one has
	      $$
	      c({\fp}^{n+1}) = c(\fp) c({\fp}^n) - {N_F(\fp)}^{2k-1} c({\fp}^{n-1}).
	      $$
\end{itemize}
\end{prop}

An analogue of the Deligne bound for Fourier coefficients of Hilbert cusp forms
was proved by Blasius \cite{DB}.
\begin{prop}\label{lem2}
Let $\bf$ be as in \propref{shi}. Then for any prime ideal 
$\fp \subset \cO_F$, one has
\begin{eqnarray*}
| c(\fp) | \leq  2N_F(\fp)^{\frac{2k - 1}{2}}.
\end{eqnarray*}
\end{prop}
An analogue of \lemref{afpn0} is also known in the context of Hilbert cusp forms (see \cite[p.10]{ASP}). Hence by applying Sato-Tate theorem for Hilbert cusp forms, we have the following lemma.
\begin{lem}\label{gunityH}
Let $\bf$ be as in \propref{shi}. Also let $\bf$ be non-CM. Then the set $P_{\bf}$ of prime ideals 
$\fp \subset \cO_F$ such that $\fp \nmid \fn \fD_F$ and $\gamma_{\bf, \fp}$ is not a root of unity has density equal to $1$.
\end{lem}

\subsection{Height of an algebraic number}
Let $K$ be a number field and $M_K$ be the set of places of $K$ normalized to 
extend the places of $\Q$. Also let $M_K^{\infty}$ be the subset of $M_K$ 
of infinite (Archimedean) places. For any prime $p$, 
let $\Q_p$ be the field of $p$-adic numbers and for each $v \in M_K$, let $K_v$ 
be the completion of $K$ with respect to the place $v$. For $v \in M_K$, let $d_v$ denotes 
the local degree at $v$ given by $d_v = [K_v : \Q_p]$, if $v \mid p$ 
and $d_v = [K_v : \R]$, if $v \in M_{K}^{\infty}$.
For $\alpha \in K\setminus \{0\}$, the usual absolute logarithmic height of $\sa$  is defined by
$$
h(\sa) = \frac{1}{[K: \Q]} \sum_{v \in M_K} d_v \log^+ |\sa|_v =  \frac{h_K(\sa)}{[K: \Q]} 
$$
where $\log^+ = \max\{\log , 0\}$, $h_K(\sa) = \sum_{v \in M_K} d_v \log^+ |\sa|_v$. It can be deduced from the definition of height that
\begin{equation}\label{hgt}
	h(\sa) = \frac{1}{[K:\Q]} \( \sum_{\sg : K \hookrightarrow \C} \log^{+}|\sg(\sa)| 
	~+~ \sum_{\fp} \max\{ 0, -\nu_{\fp}(\sa)\} \log N_K(\fp) \),
\end{equation}
where the first sum runs over the embeddings of $K$ in $\C$ and the second 
sum runs over the prime ideals of $K$. We have the following properties of $h(\cdot)$: For any non-zero algebraic numbers $\sa , \sb$ we have
\begin{itemize}
	\item $h(\sa \sb) \leq h(\sa) + h(\sb)$
	\item $h(\sa + \sb) \leq h(\sa) + h(\sb) + \log 2$ 
	\item $h(\sa^n) = |n| h(\sa), ~ n \in \Z$.
\end{itemize}

\subsection{Prerequisites from Lucas sequence}
If $\sa, \sb$ are algebraic integers such that $\sa+\sb$, $\sa \sb$ are 
non-zero co-prime rational integers and $\sa/\sb$ is not a root of unity, 
then $(\sa , \sb)$ is called a Lucas pair. Given a Lucas pair $(\sa , \sb)$,
the corresponding sequence of Lucas numbers is defined by
$$
u_n = u_n(\sa, \sb) =  \frac{{\sa}^n - {\sb}^n}{\sa - \sb} ~,~~  n= 0, 1,2, \cdots
$$
\begin{defn}
Let $(\sa , \sb)$ be a Lucas pair. A rational prime $p$ is called a
primitive divisor of $u_n$ if $p|u_n$ and $p\nmid (\sa - \sb)^2 u_1\cdots u_{n-1}$.
\end{defn}
Schinzel \cite{Sc} proved that there exists a positive absolute constant 
$n_0$ such that if $n > n_0$, then $u_n$ has a primitive divisor 
for any Lucas pair $(\sa , \sb)$ (see \cite{St2}, \cite{ Stdiv} for historical account). 
Stewart \cite{St} was the first to make the constant $n_0$ explicit by showing that
$n_0$ can be taken to be $e^{452} 2^{67}$ and this constant was further 
refined by Voutier (\cite{Vot1}, \cite{Vot2}). Finally Bilu, Hanrot and Voutier \cite{YHV} proved 
that $n_0$ can be taken to be $30$ and this is the best possible.
Schinzel \cite{Sc} in fact proved a general theorem.  Before stating his result,
we need to introduce the following definition.
\begin{defn}
Let $A, B$ be non-zero algebraic integers such that $A/B$ is not a root of unity
and $K$ be a field containing $A+B$ and $AB$.  Also let 
$$
u_n = u_n(A, B) =  \frac{{A}^n - {B}^n}{A - B} ~,~~  n= 0, 1,2, \cdots
$$
A prime ideal $\fp$ in $K$ is called a primitive divisor of $u_n$ if 
$\fp |u_n$ and $\fp \nmid AB(A - B)^2 u_1\cdots u_{n-1}$.
\end{defn}
In this set up, Schinzel showed the following

\begin{thm}\label{primdiv}
Let $A, B$ be non-zero algebraic integers
such that $A/B$ is not a root of unity and $(A,B) = \cD$
in $\Q(A, B)$.  Also let $K$ be a field containing $A+B$, $AB$
and 
$$
u_n = u_n(A, B) =  \frac{{A}^n - {B}^n}{A - B} ~,~~  n= 0, 1,2, \cdots
$$
Then $u_n$ has a primitive divisor for all $n > \max\{ n_0(d), \varphi(\fd)\}$,
where $\fd$ is the radical of the contraction of the ideal $\cD$ in $\Q\(A/B\)$, $d = [\Q\(A/B \) : \Q]$ and 
$\varphi(\fd) = N(\fd) \prod_{\fp | \fd} (1 - {N(\fp)}^{-1} )$.
\end{thm}

\begin{rmk}
Stewart \cite{St} was the first to make $n_0(d)$ explicit.  The best known result is by
Bilu and Luca \cite{YF} who proved that $n_0(d)$ can be taken to be 
$\max\{2^{d+1}, 10^{30}d^9\}$.
\end{rmk}

We also need the following lemma concerning norm of $A^n-B^n$ whose proof we furnish for the sake of completion.
\begin{lem}\label{eqnorm}	
Let $K$ be a number field and $A, B \in K$ be non-zero algebraic integers.
If $(A,B)=1$ and $A/B$ is not a root of unity, then
\begin{equation}
\log|N_K(A^n - B^n)| = h(A/B) [K : \Q] ~n \(1+O\(\frac{\log(n+1)}{n}\) \),
\end{equation}
where the implied constant depends only on $d$, the degree of $\Q(A/B)$ over $\Q$. 
\end{lem}
\begin{proof}
Let $S$ be the set of all embeddings of $K$ in $\C$. For every $\sg \in S$ 
and for all natural numbers $n$, by \cite[Corollary 4.2]{YFtri}, we have
$$
\log|{\sg(A)}^n - {\sg(B)}^n| 
~=~ 
n \log\max\Big\{|\sg(A)|,~ |\sg(B)|\Big\}  ~+~ O\( \( h\(\frac{A}{B}\) + 1 \) \log (n+1) \),
$$
where the implied constant depends only on the degree $d$ of $\Q(A/B)$ 
over $\Q$. Now summing over $\sg \in S$, we get
$$
\log|N_K(A^n - B^n)| = n \sum_{\sg \in S} \log\max\Big\{|\sg(A)|, ~|\sg(B)|\Big\} 
~+~ O\( [K:\Q] \( h\(\frac{A}{B}\) + 1 \) \log (n+1) \).
$$
Since $(A, B) =1$, from (\ref{hgt}) we have
\begin{eqnarray*}
h\( \frac{A}{B} \) 
&=&
 \frac{1}{ [K : \Q]} \(\sum_{\sg \in S} \log\max\Big\{ 1, ~\Big| \sg\(\frac{A}{B}\) \Big| \Big\} 
~+~ \log N_K (B)\)  \\
&=& 
\frac{1}{ [K : \Q]} \sum_{\sg \in S} \log\max\Big\{ |\sg(A)|, ~|\sg(B)| \Big\}.
\end{eqnarray*}
We deduce that 
$$
\log|N_K(A^n - B^n)| ~=~ h\(\frac{A}{B}\) [K : \Q] ~n \(1+O\(\(1+\frac{1}{h\(\frac{A}{B}\)}\)\frac{\log(n+1)}{n}\) \)
$$
Now the lemma follows by noting that (see \cite{ {YGH}, {ED}, {Voth}})
$$
h\( \frac{A}{B} \) ~ \geq~ \frac{1}{4d (\log^*d)^3}~,
$$ 
where $\log^* = \max \{ 1, \log  \}$.
\end{proof}

\subsection{Statement of the Frey's ABC-conjecture}
A refined version of ABC-conjecture for number fields proposed by Frey 
\cite[A-B-C-Conjecture, p. 529]{Fr}
(see also \cite[p. 33]{Fr2}, \cite[p. 5]{MW})  
is stated below.

\begin{conj}\label{conj}
Let $K$ be a number field with ring of integers $\cO_K$ and discriminant $\Delta_K$. For any $\e > 0$ and any non-zero integral ideal $\fa$ in $\cO_{K}$ there exists a constant $c(\e, \log|\Delta_K|, \log N_K(\fa))$ such that for any pair of elements $A , B \in \cO_{K} \setminus\{0\}$ which generate the ideal $\fa$, we have
$$
\max\Big\{ h_K(A), ~h_K(B), ~h_K(A-B) \Big\} ~\leq~ (1+\e) \log \( \prod_{\fp | AB(A-B)} N_K (\fp)\) 
~+~ c(\e, \log|\Delta_K|, \log N_K(\fa)).
$$	
\end{conj}
\begin{rmk}
As mentioned in \cite{Fr}, the dependence of $c(\e, \log|\Delta_K|, \log N_K(\fa))$ on $\log|\Delta_K|$ and $\log N_K(\fa))$ is linear. For simplicity, we use the notation $c(\e,\Delta_K,\fa)$ in place of $c(\e, \log|\Delta_K|, \log N_K(\fa))$. We also know that in a Dedekind domain, every non-zero ideal can be generated by at most two elements.
\end{rmk}
We will apply this conjecture to derive a lower bound for absolute norms of radicals of certain Fourier coefficients.

\medskip

\section{Number of prime factors of Fourier coefficients}

\subsection{Proof of \thmref{tau}}

For any prime $p$ and natural number $n$, the function $\tau$ satisfies the 
following binary recurrence relation
$$
\tau(p^{n+1}) = \tau(p) \tau(p^{n})- p^{11} \tau(p^{n-1}).
$$
Thus we have
\begin{equation}\label{trec}
\tau(p^n) = \frac{ \sa_{p}^{n+1} - \sb_{p}^{n+1} }{\sa_p - \sb_p} 
\quad\text{ for all } n \geq 0,
\end{equation}
where $\sa_p , \sb_p$ are the roots of the polynomial $x^2 - \tau(p)x+p^{11}$. 

From \lemref{gunity}, the set $P_{\tau}$ of primes $p$ such that $\gamma_{\tau, p}$ is not a root of unity has density 1. Fix a prime $p \in P_{\tau}$ and as before, let $\nu_{\tau, p} = \nu_p(\tau(p))$. Then by Deligne's bound, we have
$\nu_{\tau, p} \leq 5$. Set
$$
A_p = \frac{\ap}{p^{\nu_{\tau, p}}} \phantom{mm} \text{and} 
\phantom{mm} B_p = \frac{\bp}{p^{\nu_{\tau, p}}}.
$$
Then $A_p, B_p$ are roots of the polynomial $G_{\tau}(x) = x^2 - \tau(p) p^{-\nu_{\tau, p}} ~x+ p^{11- 2 \nu_{\tau, p}}$. Note that  $\nu_{p}(\tau(p)p^{-\nu_{\tau, p}})=0$ and $11-2\nu_{\tau, p}\geq 1$, since $\nu_{\tau, p} \leq 5$. Hence we have $G_{\tau}(x) \in \Z[x]$ and $(A_p , B_p)$ is a Lucas pair.
Now consider the sequence $(u_n)_{n=0}^{\infty}$ in $\Z$ defined by
$$
u_n = u_n(A_p , B_p) = \frac{A_{p}^{n} - B_{p}^{n}}{A_p - B_p}
\quad \text{ for all } n \geq 0.
$$
Using \eqref{trec}, we have $u_{n+1} =  \tau(p^n)p^{-n\nu_{\tau, p}}$.
We know that $u_t | u_n$ whenever $t | n$ (see \cite{St2}) and by \thmref{primdiv}, 
$u_n$ admits a primitive divisor  if  $n$ is sufficiently large (indeed if $n > 30$ ). Hence for large $n$, for each divisor $t$ of $n$ with $t > 30$, we get a distinct prime divisor of $u_n$.
Hence follows that $\w(u_n) \geq d(n) + O(1)$, where $d(n)$ denotes the number
of divisors of $n$ ( Similar lower bound is carried out for $\Omega(\tau(n))$ in \cite{BCOT} ). Now for any large natural number $r$, we choose
\begin{equation}\label{e-t}
n = \prod_{q \leq r \atop {q  \text{ prime}}} q.
\end{equation}
By the prime number theorem (see \cite{TA}), we then have
$$
\log n ~=~ \sum_{q \leq r \atop{q \text{ prime}}} \log q ~=~ (1+o(1))r
$$
as $r \rightarrow \infty$. Thus for any rational prime $p$ for which $\gamma_{\tau, p}$ is not a root of unity
and for sufficiently large $n$ as in~\eqref{e-t} (depending on $\e$), we get
$$
\w(\tau(p^{n-1})) \geq \w(u_n)\geq 2^{\pi(r)} +O(1) \geq 2^{(1-\e)\frac{\log n}{\log\log n}}.
$$
This completes the proof of \thmref{tau} by replacing $n-1$ with $n$.

\subsection{Proof of \rmkref{rmk1.1}} Let the notations be as before.  
Let $\e >0$ be a real number. From the proof of \thmref{tau}, we have for each prime $
p \in P_{\tau}$, the inequality
\begin{equation}\label{eqrmk}
\w(\tau(p^n)) \geq 2^{(1-\e)\frac{\log n}{\log\log n}}
\end{equation}
holds for infinitely many natural numbers $n$.

Note that $\tau(2)=-24$ and $\sa_2= -12+4i\sqrt{119}, \sb_2 = -12-4i\sqrt{119}$ 
are the roots of the  polynomial $x^2+24x+2^{11}$. Thus we have
$$
\gamma_{\tau, 2}= \frac{\sa_2}{\sb_2}= \frac{-55-3i\sqrt{119}}{64}
$$
Clearly $\gamma_{\tau, 2} \in \Q(\sqrt{-119})$ is not an algebraic integer and hence 
not a root of unity. 

This implies that $2 \in P_{\tau}$. From \eqref{eqrmk}, we get
\begin{equation}
\w(\tau(2^n)) \geq 2^{(1-\e)\frac{\log n}{\log\log n}}
\end{equation}
for infinitely many natural numbers. By choosing $m=2^n$, we deduce that
\begin{equation}
	\w(\tau(m)) \geq 2^{(1-2\e)\frac{\log\log m}{\log\log\log m}}
\end{equation}
for infinitely many natural numbers $m$.

\subsection{Proof of \thmref{EHE} }
Let $f$  be a non-CM primitive cusp form of weight $k$, level $N$ with trivial character having Fourier coefficients $\{a_f(n) ~|~ n \in \N \}$. We know that 
$a_f(n)$'s are real algebraic integers and $\K_f = \Q(\{a_f(n) ~|~n\in \N \})$ is a 
number field (see \cite{GS}). Let $\cO_{\K_{f}}$ be the ring of integers of $\K_{f}$ 
and $\w_f(a_f(n))$ be as defined in \thmref{EHE}. For any prime number $p \nmid N$
and natural number $n$, we have the following recurrence formula
$$
a_f({p}^{n+1}) = a_f(p) a_f(p^{n}) - p^{k-1} a_f(p^{n-1}).
$$
Thus for all $n \ge 0$, we have
$$
a_f(p^{n}) = \frac{ \sa_{p}^{n+1} -  \sb_{p}^{n+1}}{\ap - \bp} ~, 
$$
where $\ap, \bp$ are the roots of the polynomial $x^2 -a_f(p)x+ p^{k-1}$.
As before, from \lemref{gunity}, the set $P_f$  of primes $p \nmid N$ such that $\gamma_{f, p} = \ap / \bp$ is not a root of unity has density $1$. Fix a prime $p \in P_f$.
When $\K_f = \Q$, we proceed as in the proof of \thmref{tau} and
get the result. Now suppose that $\K_f \neq \Q$. Consider
the binary recurrent sequence $(u_n)_{n=0}^{\infty}$ in $\K_f$
defined by 
$$
u_n ~=~ u_n(\ap, \bp) ~=~ \frac{{\sa}_{p}^{n} - {\sb}_{p}^{n}}{\ap - \bp} ~, \quad n \geq 0.
$$
When $t | n$, we have $u_t | u_n$. Applying \thmref{primdiv}, we know that $u_n$ admits a 
primitive divisor for any $n$ sufficiently large (depending on $f$ and $p$). 
Choose sufficiently large $n$ as in \eqref{e-t} (depending on $\e$, $f$, $p$) and proceed as in \thmref{tau} to get
$$
\w_f(a_f(p^{n-1})) \geq 2^{(1-\e) \frac{\log n}{\log\log n}}.
$$
This completes the proof of \thmref{EHE}.

\subsection{Proof of \thmref{thm1}} 
Let $\bf$ be a non-CM primitive Hilbert cusp form of weight $2\bk$, level $\fn$ with trivial 
character having Fourier coefficients $c(\fm), ~ \fm \subset \cO_F$. 
We know that $c(\fm)$'s are real algebraic integers and 
$\K_{\bf}= \Q(\{ c(\fm) ~|~ \fm \subset \cO_F\})$ is a number field (see \cite{Sh}).
Let $\cO_{\K_{\bf}}$ is the ring of integers of $\K_{\bf}$,
For any prime ideal $\fp$ in $\cO_F$ with $\fp \nmid \fn\fD_F$ and for any natural number
$n$, we have
$$
c({\fp}^{n+1}) = c(\fp) c({\fp}^n) - {N_F(\fp)}^{2k-1} c({\fp}^{n-1}).
$$
Hence for any integer $n \ge 0$, we have
$$
c({\fp}^n) = \frac{{\sa}_{\fp}^{n+1} - {\sb}_{\fp}^{n+1}}{\afp - \bfp}~,
$$
where $\afp, \bfp$ are the roots of the polynomial $x^2 - c(\fp)x + {N_F(\fp)}^{2k-1}$.

From \lemref{gunityH}, the set of prime ideals $\fp \subset \cO_F$ such that $\fp \nmid \fn \fD_F$ and $\gamma_{\bf, \fp}$ is not a root of unity has density equal to $1$. Fix a prime ideal $\fp \in P_{\bf}$ and let $N_{F}(\fp) = p^{\ell}$.
Consider the binary recurrent sequence $(u_n)_{n=0}^{\infty}$ in $\K_{\bf}$
defined by 
$$
u_n ~=~ u_n(\afp, \bfp) ~=~  \frac{{\sa}_{\fp}^{n} - {\sb}_{\fp}^{n}}{\afp - \bfp}
 \quad \text{ for all }n \geq 0.
$$
When $t | n$, we have $u_t | u_n$. Applying \thmref{primdiv}, we know that $u_n$ admits a 
primitive divisor for any $n$ sufficiently large (depending on $\bf$ and $\fp$). 
Choose sufficiently large $n$ as in \eqref{e-t} (depending on $\e$, $\bf$, $\fp$) and proceed as in \thmref{tau} to get
$$
\w_{\bf} (c({\fp}^{n-1})) \geq 2^{(1-\e) \frac{\log n}{\log\log n}}.
$$
This completes the proof of \thmref{thm1}.

\medskip

\section{Radical of Fourier coefficients}
Let $p_i$ denotes the $i$-th rational prime number. For any non-zero integer $n$, 
we have the following inequality
$$
Q(n) \geq \prod_{p < p_{\w(n)}} ~p.
$$
Let $(a_n)_{n=1}^{\infty}$ be a sequence of integers such that 
$\w(a_n) \rightarrow \infty$ as $n \rightarrow \infty$. We have
$$
\log(Q(a_n)) \geq   \sum_{p < p_{\w(a_n)} } \log p.
$$
By the prime number theorem (see \cite{TA}), it follows that 
\begin{equation}
\label{eq1}
\log(Q(a_n)) \geq   \sum_{p < p_{\w(a_n)} } \log p ~=~ (1+o(1))~~ p_{\w(a_n)} ~
=~(1+o(1)) ~~\w(a_n) \log\w(a_n)
\end{equation}
as $n \rightarrow \infty$.

Let $K$ be a number field with ring of integers $\cO_K$. We list all 
prime ideals in $\cO_K$ in the increasing order of their absolute norm (if there are
several ideals with the same norm, choose any ordering among them).
Let $\fp_i$ denotes the $i$-th prime ideal in $\cO_{K}$. For an algebraic number 
$\sa \in \cO_K$, let $\w_K(\sa)$ be the number of distinct prime ideals dividing $\sa \cO_K$  
with the convention that $\w_K(\sa) =~0$ if $\sa=0$ or is a unit in $\mathcal{O}_K$.

We have the following inequality
$$
Q_K(\sa) \geq \prod_{i=1}^{\w_K(\sa)} N_K(\fp_i)
$$

Let $(\sa_n)_{n=1}^{\infty}$ be a sequence of algebraic integers in 
$K$ such that $\w_K(\sa_n) \rightarrow \infty$ as $n \rightarrow \infty$.

As before, by applying Landau's prime ideal theorem, we deduce that
\begin{equation}
\label{eq2}
\log Q_K(\sa_n) \geq (1+o(1)) \w_K(\sa_n) \log\w_K(\sa_n).
\end{equation}
as $n \rightarrow \infty$.

\subsection{Proof of \corref{cor1} and \corref{cor2}}
\corref{cor1} and \corref{cor2} now follows from equations \eqref{eq1}, \eqref{eq2} 
and \thmref{tau}, \thmref{EHE} and  \thmref{thm1}.

\smallskip
\section{Frey's ABC-conjecture and Radical of $A^n-B^n$}

\subsection{Proof of \thmref{proprad-int}}

By hypothesis, $A, B \in \cO_K$ be such that $(A,B) = 1$,  $A/B$ is not a root 
of unity and $h(A) \geq h(B)$.  Also let $\e > 0$ be arbitrary and $n \in \N$. 
By Frey's refined ABC-conjecture, we have
$$
\max\Big\{h_K(A^n), ~h_K(B^n), ~h_K(A^n -B^n)\Big\} \leq 
~(1+\e)~ \log \( \prod_{\fp | AB(A^n - B^n)} N_K (\fp)\) ~+~ c_1(\e, \Delta_K).
$$
Since $h_K(A^n) = n h_K(A)$, we get
$$
e^{n h_K(A)} ~\leq ~ e^{c_1(\e, \Delta_K)}  (Q_K(AB))^{1+\e} (Q_K(A^n -B^n))^{1+\e}.
$$
Hence we deduce that
\begin{equation}\label{eqA}
Q_K(A^n -B^n) \geq \frac{e^{-c_1(\e, \Delta_K)}}{Q_K(AB)} ~ e^{(1-\e) [K : \Q] h(A) n}
\end{equation}
for all natural numbers $n$. Here $c_1(\e, \Delta_K)$ is a positive constant 
depending only on $\e$ and the discriminant $\Delta_K$ of the field $K$.

\subsection{Proof of \rmkref{rmkrad-int}}

Let the notations be as in \thmref{proprad-int}. We note that 
$$
Q_K\(\frac{A^n-B^n}{A-B} \) \geq \frac{Q_K(A^n - B^n)}{Q_K(A-B)}.
$$
From \eqref{eqA}, we get
\begin{equation}\label{eqAL}	
Q_K\(\frac{A^n-B^n}{A-B} \) \geq \frac{e^{-c_1(\e, \Delta_K)}}{Q_K( AB(A-B)^2 )} 
~ e^{(1-\e) [K : \Q] h(A) n}
\end{equation}
Suppose that $K_0$ is a number field containing $AB, A+B$ and hence contains the sequence
 $\frac{A^n - B^n}{A-B}$ for all $n \in \N$. Now choose $K = K_0 (A)$. 
 For any algebraic integer $C \in K_0$, we have 
\begin{equation}\label{eqradG}
Q_K(C) \leq \( Q_{K_0}(C) \)^{[K : K_0]}.
\end{equation}
From \eqref{eqAL} and \eqref{eqradG}, we conclude that for all natural numbers $n$,
\begin{equation}\label{eqL0A}
Q_{K_0}\(\frac{A^n-B^n}{A-B} \) \geq \frac{e^{-c_1(\e, \Delta_K)}}
{Q_{K_0}( AB(A-B)^2 )} ~ e^{(1-\e) [K_0 : \Q] h(A) n}.
\end{equation}

\medskip
\subsection{Proof of \thmref{proprad-alg}}
Let $A, B$ be non-zero algebraic integers in $\cO_K$ 
such that $A/B$ is not a root of unity and $(A, B)=\cD$.
Also let $K_1$ be an extension of $K$ of smallest degree such that $\cD$ is 
a principal ideal in $K_1$, say $\cD \cO_{K_1} = (D)$. Set 
$A_1 = \frac{A}{D}, ~B_1 = \frac{B}{D}$ and hence $(A_1, B_1) = 1$
in $\cO_{K_1}$. 

Let $\e > 0$ be arbitrary and $n \in \N$. By Frey's refined ABC-conjecture, we have
$$
\max\Big\{ h_{K_1}(A_{1}^{n}),~ h_{K_1}(B_{1}^{n}),~ h_{K_1}(A_{1}^{n} -B_{1}^{n})\Big\} 
~\leq 
(1+\e) ~\log \( \prod_{\fp | A_1 B_1 (A_{1}^{n} -B_{1}^{n})} N_{K_1} (\fp)\) ~+~ c_1(\e, \Delta_{K_1}).
$$
From \eqref{hgt}, $h_{K_1}(A_{1}^{n} -B_{1}^{n}) \geq \log |N_{K_1}(A_{1}^{n} -B_{1}^{n})|$, thus we get
$$
|N_{K_1}(A_{1}^{n} -B_{1}^{n})| \leq e^{c_1(\e, \Delta_{K_1})}  
(Q_{K_1}(A_1 B_1))^{1+\e} (Q_{K_1}(A_{1}^{n} -B_{1}^{n}))^{1+\e}
$$
From \lemref{eqnorm}, we deduce that
\begin{equation}\label{Q1}
Q_{K_1}(A_{1}^{n} -B_{1}^{n}) \geq \frac{e^{-c_1(\e, \Delta_{K_1})}}
{Q_{K_1}(A_1 B_1)} ~ e^{(1-2\e) [K_1 : \Q] h\(\frac{A_1}{B_1}\) n}
\end{equation}
for all sufficiently large natural numbers $n$ depending only on $\e$ and degree of $\Q(A/B)$ over $\Q$.
Since $A^n - B^n = {D}^n (A_{1}^{n} -B_{1}^{n})$ and $AB = D^2 A_1B_1$, thus we get
\begin{equation}\label{Q2}
\begin{split}
& Q_{K_1}(A_{1}^{n} -B_{1}^{n}) \leq Q_{K_1}(A^n -B^n) \leq (Q_K(A^n -B^n))^{[K_1 : K]} ~,\\
& Q_{K_1}(A_1 B_1) \leq Q_{K_1}(AB) \leq (Q_{K}(AB))^{[K_1 : K]} .
\end{split}
\end{equation}
Hence we conclude that
\begin{equation}\label{Q3}
Q_{K}(A^n -B^n) \geq \frac{e^{-c_1(\e, \Delta_{K_1})}}{Q_K(AB)}
~ e^{(1-2\e) [K : \Q] h\(\frac{A}{B}\) n}
\end{equation}
for all sufficiently large natural numbers $n$ depending only 
on $\e$ and degree of $\Q(A/B)$ over $\Q$.

\subsection{Proof of \rmkref{rmkrad-alg}}

Let the notations be as in \thmref{proprad-alg}. From \eqref{Q1},  we have 
\begin{equation}\label{Q4}
Q_{K_1} \(\frac{A_1^n -B_1^n}{A_1 - B_1}\)  
\geq \frac{e^{-c_1(\e, \Delta_{K_1})}}{Q_{K_1}(A_1B_1(A_1- B_1)^2)} ~ e^{(1-2\e) [K_1 : \Q]h\(\frac{A}{B}\) n} 
\end{equation}
for all sufficiently large natural numbers $n$ depending only on $\e$ and degree 
of $\Q(A/B)$ over $\Q$. 
Since 
\begin{equation}
\frac{A^n -B^n}{A - B} = D^{n-1} \frac{A_1^n -B_1^n}{A_1 - B_1} \quad \text{and} \quad
 AB(A-B)^2 = D^4 A_1 B_1 (A_1-B_1)^2.
\end{equation}
Thus we get
\begin{equation}\label{Q5}
\begin{split}
&Q_{K_1} \(\frac{A_1^n -B_1^n}{A_1 - B_1} \) \leq \( Q_K\(\frac{A^n -B^n}{A - B} \)\)^{[K_1:K]} \\
&Q_{K_1}(A_1 B_1 (A_1-B_1)^2) \leq \(Q_K(AB(A-B)^2)\)^{[K_1:K]}.
\end{split}
\end{equation}
From \eqref{Q4} and \eqref{Q5}, we get
\begin{equation}\label{eqLKB}
Q_K \(\frac{A^n -B^n}{A - B}\)  
~\geq~ 
\frac{e^{-c_1(\e, \Delta_{K_1})}}{Q_{K}(AB(A - B)^2)} ~ e^{(1-2\e) [K : \Q] 
h\(\frac{A}{B}\) n}.
\end{equation}
Suppose that $K_0$ is a number field containing $AB, A+B$ and hence contains 
the sequence $\frac{A^n - B^n}{A-B}$ for all $n \in \N$. Let $K= K_0(A)$ 
and $K_1$ be an extension of $K$ of smallest degree such that  
$(A, B)$ is principal ideal in $K_1$. From \eqref{eqradG} and \eqref{eqLKB}, 
we conclude that
\begin{equation}\label{eqL0B}
Q_{K_0}\(\frac{A^n -B^n}{A-B}\) 
~\geq~ 
\frac{e^{-c_1(\e, \Delta_{K_1})}}{Q_{K_0}(AB(A-B)^2)} 
~ e^{(1-2\e) [K_0 : \Q] h\(\frac{A}{B}\) n}
\end{equation}
for all sufficiently large natural numbers $n$ depending only on $\e$ 
and degree of $\Q(A/B)$ over $\Q$.

\medskip

\section{Frey's ABC-conjecture and Radical of Fourier coefficients}
In this section  assuming Frey's refined ABC-conjecture, we derive 
lower bound on the norm of radical of certain Fourier coefficients of elliptic modular forms as well as Hilbert modular forms.

\subsection{Proof of \thmref{corrad-int} for $\tau$ function}
Let the notations be as in \S1, \S2 , \S3 and \S4.
Let $p$ be a rational prime such that $\gamma_{\tau, p}$ is not a root of unity. 
For any natural number $n$, we have
\begin{equation}\label{tAp}
\tau(p^n) 
~=~ \frac{{\sa}_{p}^{n+1} - {\sb}_{p}^{n+1}}{\sa_p - \sb_p}
~=~ p^{n \nu_{\tau, p}} \cdot \frac{A_{p}^{n+1} - B_{p}^{n+1}}{A_p - B_p}~,
\end{equation}
where $A_p, B_p$ are the roots of the irreducible polynomial $x^2 - \tau(p) 
p^{- \nu_{\tau, p}}x + p^{11 -2 \nu_{\tau, p}}$ over $\Z$. From \eqref{hgt}, we have 
\begin{equation}\label{htAp}
h(A_p) = h(B_p) = \frac{ \log|A_p|+ \log|B_p|}{2} =\log|A_p|= \( \frac{11}{2} - \nu_{\tau, p}\) \log p.
\end{equation}
From \rmkref{rmkrad-int} and \eqref{htAp}, we have
\begin{equation}\label{eqradAp}
Q\( \frac{A_{p}^{n+1} - B_{p}^{n+1}}{A_p - B_p} \) 
\geq  \frac{c(\e, \Delta_{\tau, p})}{Q(A_p B_p (A_p-B_p)^2)} 
~ p^{(1-\e)\( \frac{11}{2} - \nu_{\tau, p}\) (n+1)}
\end{equation}
where $c(\e, \Delta_{\tau , p})$ is a positive constant depending on 
$\e$ and the discriminant of the field $\Q(\ap)$. 

We will show by induction that
\begin{equation}\label{tauv}
\nu_p(\tau(p^n)) =  n~\nu_{\tau, p}~,
\end{equation}
for all natural numbers $n$. For $n=1$, it follows from definition. Let $n >1$ and assume that 
$$
\nu_p(\tau(p^m)) =  m~\nu_{\tau, p} \phantom{mm} \text{for} \phantom{mm} m= 1, 2,\cdots, n-1.
$$ 
Since $\tau(p^n) = \tau(p) \tau(p^{n-1}) - p^{11} \tau(p^{n-2})$, by induction 
hypothesis, we get 
$$
\nu_p(\tau(p)\tau(p^{n-1})) = n \nu_{\tau, p} \phantom{mm} \text{and} 
\phantom{mm} \nu_p( p^{11} \tau(p^{n-2})) = 11+(n-2)\nu_{\tau, p}.
$$ 
By Deligne's bound, we have $\nu_{\tau, p} \leq 5$. Thus 
we get $\nu_p(\tau(p)\tau(p^{n-1})) <  \nu_p( p^{11} \tau(p^{n-2}))$ and 
hence it follows that $\nu_p(\tau(p^n)) =  n~\nu_{\tau, p}$. Thus by induction we conclude that
$$
\nu_p(\tau(p^n)) =  n~\nu_{\tau, p} ~,
$$
for all natural numbers $n$. 

Now suppose that $\nu_{\tau, p} = 0$. Then we have $A_p B_p = p^{11}$ 
and $(A_p - B_p)^2 = \tau(p)^2 - 4 p^{11}$. Thus we get
\begin{equation}\label{QtAB}
Q(A_p B_p (A_p - B_p)^2) = p \cdot Q({\tau(p)}^2 - 4 p^{11}).
\end{equation}
From \eqref{tAp}, \eqref{eqradAp} and \eqref{QtAB}, we deduce that
$$
Q(\tau(p^n)) \geq \frac{c(\e, \Delta_{\tau, p})}{Q({\tau(p)}^2 - 4 p^{11})}~~ p^{\frac{11}{2}(1-\e) (n+1) - 1 }
$$
for all natural numbers $n$. 

If $\nu_{\tau, p} \neq 0$, then we have 
$$
Q((A_p-B_p)^2) = Q( \tau(p)^2 p^{-2\nu_{\tau, p}}  - 4p^{11-2\nu_{\tau, p}}) 
~=~
 \frac{Q({\tau(p)}^2 - 4 p^{11})}{p}~.
$$
Thus we get
\begin{equation}\label{QtAB2}
Q(A_p B_p (A_p - B_p)^2) = Q({\tau(p)}^2 - 4 p^{11}).
\end{equation}
From \eqref{tAp}, \eqref{eqradAp}, \eqref{tauv} and \eqref{QtAB2}, we deduce that
$$
Q(\tau(p^n)) \geq \frac{c(\e, \Delta_{\tau, p})}{Q({\tau(p)}^2 - 4 p^{11})}~~ p^{(1-\e)  \(\frac{11}{2}- \nu_{\tau, p}\) (n+1) + 1 }
$$
for all natural numbers $n$.

\medskip
\subsection{Proof of \thmref{corrad-int} for non-CM primitive cusp forms}
Let $f$  be a non-CM primitive cusp form of weight $k$, level $N$ with trivial character having integer Fourier coefficients $\{a_f(n) ~|~ n \in \N \}$. For any prime $p \nmid N$ and any natural number 
$n$, we have
\begin{equation}\label{fLap}
a_f(p^{n}) = \frac{ \sa_{p}^{n+1} -  \sb_{p}^{n+1}}{\ap - \bp}. 
\end{equation}
As before $\nu_{f, p} = \nu_p(a_f(p))$. By Deligne's bound, we 
have $\nu_{f, p} \leq k/2 -1$ for $p > 3$. 
Now onwards, assume that $p$ is a prime such that $p \nmid N , p>3$ and $\gamma_{f, p}$ is not a root of unity. Set
$$
A_p = \frac{\sa_p}{p^{\nu_{f, p}}} \phantom{mm} \text{and} 
\phantom{mm} B_p = \frac{\sb_p}{p^{\nu_{f, p}}}~,
$$
where $A_p , B_p$ are the roots of the irreducible polynomial 
$x^2-a_f(p) p^{-\nu_{f, p}} ~x+ p^{k-1-2\nu_{f, p}}$ over $\Z$. 
From \eqref{hgt}, we have
\begin{equation}\label{hfAp}
h(A_p) = h(B_p) = \frac{ \log|A_p|+ \log|B_p|}{2} =\log|A_p|= \( \frac{k-1}{2} - \nu_{f, p}\) \log p.
\end{equation}
From \eqref{fLap}, we have
\begin{equation}
a_f(p^n) = p^{n\nu_{f, p} } \cdot \frac{A_{p}^{n+1} - B_{p}^{n+1}}{A_p - B_p} .
\end{equation}
As before, we can show that  
$$
\nu_p(a_f(p^n)) = n~\nu_{f, p}
$$
for all natural numbers $n$.
Now proceeding as in the case of 
$\tau$ function, we deduce that
$$
Q(a_f(p^n)) \geq \frac{c(\e, \Delta_{f, p})}{Q(a_f(p)^2 - 4p^{k-1})}
~~ p^{(1-\e)\(\frac{k-1}{2} - \nu_{f, p}\) (n+1)~+~\delta_{f, p}} 
$$
for all natural numbers $n$. Here $\delta_{f, p}$ be as in \thmref{corrad-int}.

\medskip
\subsection{Proof of \thmref{corrad-int} for non-CM primitive Hilbert cusp forms}
Let $\bf$ be a non-CM primitive Hilbert cusp form of weight $2\bk$, level $\fn$ with 
trivial character having integer Fourier coefficients $c(\fm), ~ \fm \subset \cO_F$. For any 
prime ideal $\fp$ in $\cO_F$ with $\fp \nmid \fn\fD_F$ and for any natural number $n$, we have
\begin{equation}\label{bfLap}
c({\fp}^n) = \frac{{\sa}_{\fp}^{n+1} - {\sb}_{\fp}^{n+1}}{\afp - \bfp}.
\end{equation}
As before  $N_F(\fp) = p^{\ell}$ and $\nu_{\bf, \fp} = \nu_p(c(\fp))$. By 
Blasius's bound, we have 
$\nu_{\bf, \fp} \leq \(\frac{2k-1}{2}\)\ell - \frac{1}{2}$ for $p > 3$. 
Now onwards, let us assume that $\fp$ is a prime ideal in $\cO_F$ 
such that $\fp \nmid \fn\fD_F, p> 3$ and $\gamma_{\bf, \fp}$ is not a root of unity. Set
$$
A_{\fp} = \frac{\sa_{\fp}}{p^{\nu_{\bf, \fp}}}
 \phantom{mm}\text{and} \phantom{mm} B_{\fp} = \frac{\sb_{\fp}}{p^{\nu_{\bf, \fp}}},
$$
where $A_{\fp}, B_{\fp}$ are the roots of the irreducible 
polynomial $x^2-c(\fp) p^{-\nu_{\bf, \fp}} ~x+ p^{(2k-1)\ell - 2\nu_{\bf, \fp}}$
over $\Z$. From \eqref{hgt}, we have
\begin{equation}\label{hbfAp}
h(A_{\fp}) = h(B_{\fp}) =\log|A_{\fp}|= \( \frac{2k-1}{2} - \frac{\nu_{\bf, \fp}}{\ell}\) \log N_F(\fp).
\end{equation}
From \eqref{bfLap}, we have
\begin{equation}
c({\fp}^n) = p^{n \nu_{\bf, \fp}} \cdot \frac{A_{\fp}^{n+1} - B_{\fp}^{n+1}}{A_{\fp} - B_{\fp}}.
\end{equation}
As before we can show that 
$$
\nu_p(c({\fp}^n)) = n~\nu_{\bf, \fp}
$$
for all natural numbers $n$. Now proceeding as in the case of 
$\tau$ function, we deduce that
$$
Q(c({\fp}^n)) \geq \frac{c(\e, \Delta_{\bf, \fp})}{Q(c(\fp)^2 - 4 N_F(\fp)^{2k-1})}~
~ {N_F(\fp)}^{(1-\e)\(\frac{2k-1}{2} - \frac{\nu_{\bf, \fp}}{\ell} \)(n+1)
~+~ \frac{\delta_{\bf, \fp}}{\ell}}
$$
for all natural numbers $n$. Here $\delta_{\bf, \fp}$ be as in \thmref{corrad-int}.

\subsection{Proof of \thmref{corradalg} for non-CM primitive cusp forms}
Let $f$ be as in \thmref{EHE}. Let $p$ be a prime such that $p \nmid N$ and $\gamma_{f, p}$ is not a root of unity. For any natural number $n$, we have
\begin{equation}\label{fLap2}
a_f(p^{n}) = \frac{ \sa_{p}^{n+1} -  \sb_{p}^{n+1}}{\ap - \bp}. 
\end{equation}
Let $\e > 0$ be arbitrary. As before $\gamma_{f, p} = \sa_{p}/\sb_{p}$.
Using \eqref{fLap2} and applying \rmkref{rmkrad-alg} with
$K_0 = \K_f$, we get
\begin{equation}
	Q_f(a_f(p^n)) \geq c(\e, f, p) ~~ e^{(1-\e) [\K_f : \Q] h(\gamma_{f, p}) n}
\end{equation}
for all sufficiently large natural numbers $n$ depending on $\e$ and the degree 
of $\Q(\gamma_{f, p})$ over $\Q$. Here $c(\e, f, p)$ is a positive 
constant depending on $\e, f$ and $p$.

\medskip
\subsection{Proof of \thmref{corradalg} for non-CM primitive Hilbert cusp forms}

Let $\bf$ be as in \thmref{thm1}. 
Let $\fp$ be a prime ideal in $F$ such that $\fp \nmid \fn\fD_F$ and 
$\gamma_{\bf, \fp}$ is not a root of unity.
For any natural number $n$, we have
\begin{equation}\label{bfLap2}
c({\fp}^n) = \frac{{\sa}_{\fp}^{n+1} - {\sb}_{\fp}^{n+1}}{\afp - \bfp}.
\end{equation}
Let $\e > 0$ be arbitrary. As before $\gamma_{\bf, \fp} = \sa_{\fp} / \sb_{\fp}$. 
Using \eqref{bfLap2} and applying \rmkref{rmkrad-alg} with
$K_0 = \K_{\bf}$, we get
\begin{equation}
Q_{\bf}\(c({\fp}^n)\) \geq c(\e, \bf, \fp) ~~e^{(1-\e) [\K_{\bf}: \Q]h(\gamma_{\bf, \fp}) n}
\end{equation}
for all sufficiently large natural numbers $n$ depending on $\e$ and the degree 
of $\Q(\gamma_{\bf, \fp})$ over $\Q$. Here $c(\e, \bf, \fp)$ is a 
positive constant depending on $\e, \bf$ and $\fp$.

\section*{Acknowledgments}
The author would like to thank Sanoli Gun and Purusottam Rath for their 
support and guidance. The author would also like to thank the referee 
for helpful suggestions. The author would also like to acknowledge the support
 of DAE number theory plan project.

\end{document}